\documentclass[12pt]{amsproc}
\newtheorem{theorem}{\sc Theorem}[section]
\newtheorem{lemma}[theorem]{\sc Lemma}
\newtheorem{proposition}[theorem]{\sc Proposition}
\newtheorem{corollary}[theorem]{\sc Corollary}

\begin{document}

\author{Jo\~ao Azevedo}
\address{Department of Mathematics, University of Brasilia\\
Brasilia-DF \\ 70910-900 Brazil}
\email{J.P.P.Azevedo@mat.unb.br}

\author{Pavel Shumyatsky}
\address{Department of Mathematics, University of Brasilia\\
Brasilia-DF \\ 70910-900 Brazil}
\email{pavel@unb.br}
\thanks{Supported by CNPq and FAPDF}

\title[Compact groups]{Compact groups with high commuting probability of monothetic subgroups}

\begin{abstract} If $H$ is a subgroup of a compact group $G$, the probability that a random element of $H$ commutes with a random element of $G$ is denoted by $Pr(H,G)$. Let $\langle g\rangle$ stand for the monothetic subgroup generated by an element $g\in G$ and let $K$ be a subgroup of $G$. We prove that $Pr(\langle x\rangle,G)>0$ for any $x\in K$ if and only if $G$ has an open normal subgroup $T$ such that $K/C_K(T)$ is torsion. In particular, $Pr(\langle x\rangle,G)>0$ for any $x\in G$ if and only if $G$ is virtually central-by-torsion, that is, there is an open normal subgroup $T$ such that $G/Z(T)$ is torsion. We also deduce a number of corollaries of this result.
\end{abstract}

\maketitle

\section{Introduction}

Let $K$ be a subgroup of a finite group $G$. The relative commutativity degree $Pr(K,G)$ of $K$ in $G$ was defined in \cite{erl} as the probability that an element of $G$ commutes with an element of $K$. The main result of \cite{ds} says that if $Pr(K,G)\geq\epsilon>0$, then there is a normal subgroup $T\leq G$ and a subgroup $B\leq K$ such that the indices $[G:T]$ and $[K:B]$ and the order of the commutator subgroup $[T,B]$ are $\epsilon$-bounded. In the case where $K=G$ this is a well known theorem of P. M. Neumann \cite{pmneumann}. Throughout the article we use the expression ``$(a, b, \dots )$-bounded" to mean that a quantity is bounded from above by a number depending only on the parameters $a,b,\dots$. If $B$ and $T$ are subgroups of a group $G$, we denote by $[T,B]$ the subgroup generated by all commutators $[t,b]$ with $t\in T$ and $b\in B$.

There is a natural manner to interpret the above probabilistic results to compact groups by considering the normalized Haar measure (cf. \cite{gustafson}). In particular, the aforementioned result of \cite{ds} was extended to compact groups in \cite{azeshu}. Throughout this paper, compact groups are Hausdorff topological spaces. By a subgroup of a topological group we mean a closed subgroup unless explicitly stated otherwise. If $S$ is a subset of a topological group $G$, then we denote by $\langle S\rangle$ the subgroup (topologically) generated by $S$. A subgroup of $G$ is monothetic if it is generated by a single element. 

The present paper deals with compact groups $G$ having a subgroup $K$ such that $Pr(\langle x\rangle,G)>0$ for any $x\in K$. Note that this condition is satisfied whenever the subgroup $K$ is torsion. More generally, the condition is satisfied whenever the image of $K$ in $G/Z(G)$ is torsion. The next theorem shows that an ``almost converse"  to the latter statement also holds.

\begin{theorem}\label{main1} Let $K$ be a subgroup of a compact group $G$. Then $Pr(\langle x\rangle,G)>0$ for any $x\in K$ if and only if $G$ has an open normal subgroup $T$ such that $K/C_K(T)$ is torsion.
\end{theorem}

A ``bounded" version of this result was earlier obtained in \cite{azeshu}.
\smallskip

{\it Let $\epsilon>0$ and $K$ be a subgroup of a compact group $G$ such that $Pr(\langle x\rangle,G)>\epsilon$ for any $x\in K$. Then there is an $\epsilon$-bounded number $e$ and an open normal subgroup $T\leq G$ such that the index $[G:T]$ and the order of $[K^e,T]$ are $\epsilon$-bounded.}
\smallskip

Throughout, if $e$ is a positive integer, $H^e$ denotes the subgroup generated by all $e$th powers of elements of a group $H$. We remark that if in the above result $|[K^e,T]|=f$, then $[K^{ef!},T]=1$. Indeed, since $K^e$ normalizes the subgroup $[K^e,T]$, it follows that $K^{e(f-1)!}$ centralizes $[K^e,T]$. If $x\in K^{e(f-1)!}$ and $t\in T$, we have $[x^f,t]=[x,t]^f=1$. Hence, $[K^{ef!},T]=1$ and so $K/C_K(T)$ has $\epsilon$-bounded exponent.
 \smallskip

 An immediate corollary of Theorem \ref{main1} is that if $G$ is a compact group such that $Pr(\langle x\rangle,G)>0$ for any $x\in G$, then $G$ is virtually central-by-torsion, that is, there is an open normal subgroup $T$ such that $G/Z(T)$ is torsion.

Compact torsion groups have attracted significant interest in the past. Wilson showed that any such group possesses a characteristic series of finite length each of whose factors either is a pro-$p$ group for some prime $p$ or is isomorphic to a Cartesian product of isomorphic finite simple groups \cite{wilson83}. This has enabled Zelmanov to prove that compact torsion groups are locally finite, that is, any finite subset in such a group generates a finite subgroup \cite{ze}. The problem whether compact torsion groups have finite exponent remains open for many years (cf \cite[p. 70]{hewitt-ross}). Recall that $G$ has finite exponent $e$ if $G^e = 1$ and $e$ is the least positive number with this property. We remark that if indeed compact torsion groups have finite exponent, then under the hypotheses of Theorem \ref{main1} there is $\epsilon>0$ such that $Pr(\langle x\rangle,G)>\epsilon$ for any $x\in K$ and so there is a number $e$ such that $[K^e,T]=1$.

Other ways for dealing with probability in infinite groups are considered in \cite{amv, atvv, tointon}.

\section{Proof of Theorem \ref{main1}}

We start this section by quoting the following proposition obtained in \cite{azeshu}. As mentioned in the introduction, the result for finite groups was obtained in \cite{ds}.

\begin{proposition}\label{ds}
Let $\epsilon > 0$ and let $G$ be a compact group having a subgroup $K$ such that $Pr(K,G) \geq \epsilon$. Then there is a normal subgroup $T \leq G$ and a subgroup $B \leq K$ such that the indices $[G:T]$ and $[K:B]$ and the order of $[T,B]$ are $\epsilon$-bounded. 
\end{proposition}

Recall that $FC(G)$ denotes the set of all elements $x$ of a group $G$ such that $C_G(x)$ has finite index. Note that in general $FC(G)$ is an abstract subgroup and there are easy examples of compact groups $G$ in which $FC(G)$ is not closed.
\begin{lemma}\label{lemma1}
Let $G$ be a compact group and $x\in G$ be an element such that $Pr(\langle x\rangle, G) > 0$. Then there is a positive integer $e$ such that $x^e$ is contained in an abstract torsion-free abelian normal subgroup $N\leq FC(G)$.
\end{lemma}
\begin{proof}
Since $Pr(\langle x\rangle, G) > 0$, an application of Proposition \ref{ds} shows that there is a positive integer $e_1$ such that $x^{e_1}$ has finitely many conjugates. Let $\{y_1, y_2, \dots, y_k\}$ be the conjugacy class of $x^{e_1}$ and let $L$ be the abstract subgroup generated by $y_1, y_2, \dots, y_k$. Since $[G:C_G(y_i)] = k$ for $i = 1, 2, \dots, k$, it follows that the intersection $\bigcap _{i=1}^k C_L(y_i) = Z(L)$ has index at most $k^k$ in $L$. Set $e_2 =  k^k$ and note that the power $x^{e_1e_2}$ belongs to $Z(L)$. Let $M$ be the abstract subgroup generated by $y_1^{e_2}, y_2^{e_2}, \dots, y_k^{e_2}$. Obviously, $M$ is a finitely generated abelian group. Thus, we conclude that the torsion part of $M$ is finite and denote by $e_3$ the exponent of the torsion subgroup of $M$. Write $e = e_1e_2e_3$ and observe that the minimal abstract normal subgroup $N$ containing $ x^e$ is torsion-free, abelian, and normal. Since $N$ is generated by finitely many $FC$-elements, it follows that $N\leq FC(G)$. The proof is complete. 
\end{proof}

Now we are able to prove Theorem \ref{main1}

\begin{proof}[Proof of Theorem 1.1]   Suppose first that $G$ has an open normal subgroup $T$, say of index $i$, such that $K/C_K(T)$ is torsion. Choose $x\in K$ and note that the probability that a random element of $\langle x\rangle$ centralizes $T$ is at most $\frac{1}{j}$, where $j$ is the order of the image of $\langle x\rangle$ in $K/C_K(T)$. On the other hand, the probability that a random element of $G$ belongs to $T$ is $\frac{1}{i}$ and so we deduce that $Pr(\langle x\rangle, G)\geq\frac{1}{ij}$. In particular,  $Pr(\langle x\rangle, G)>0$ for every $x \in K$.

We now need to prove the other part of the theorem, that is, assuming that $G$ is a compact group having a subgroup $K$ such that $Pr(\langle x\rangle, G) > 0$ for every $x \in K$ we need to show that $G$ has an open normal subgroup $T$ such that $K/C_K(T)$ is torsion. 

For an element $x \in K$ write $e(x)$ to denote the least positive integer $e$ satisfying the conclusion of Lemma \ref{lemma1}.  Also, define $A_x$ as the abstract subgroup generated by the conjugacy class of $x^{e(x)}$. Thus, $A_x$ is torsion-free and abelian.

Let $x$ and $y$ be arbitrary elements of $K$. The product $A_xA_y$ is an abstract $FC$-subgroup of $G$, and so by a well-known property of $FC$-groups the commutator subgroup of $A_xA_y$ is torsion (see \cite[Theorem 14.5.9]{robinson}). On the other hand, the commutator subgroup of $A_xA_y$ is contained in the intersection $A_x\cap A_y$, which is torsion-free. It follows that $A_xA_y$ is abelian. Since this happens for any $x,y\in K$, we conclude that the product $\prod_{x\in K}A_x$ is abelian.

Let $N$ be the topological closure of the abstract subgroup $\prod_{x\in K}A_x$. We see that $N$ is an abelian normal subgroup of $G$ and $KN/N$ is torsion. Set $M = N \cap K$ and, for positive integers $k,s$, define $$M_{k,s} = \{x \in M \, | \, x^k \, \text{has at most $s$ conjugates in $G$}\}.$$ Lemma \ref{lemma1} shows that the sets $M_{k,s}$ cover $M$. We also note that the sets are closed (see in particular \cite[Lemma 2.8]{azeshu}). The Baire Category Theorem \cite[p. 200]{kelley} ensures that at least one of the above sets has non-empty interior. Therefore, for some positive integers $k,s$, there is an open subset $V$ of $M$ such that $x^k$ has at most $s$ conjugates for every $x\in V$. As $M$ is compact, there is a finite subcover $M=\bigcup_{i=1}^n x_iV$ of the cover $M = \bigcup_{x \in M} xV$. Moreover, by Lemma \ref{lemma1}, for each $i = 1, 2, \dots, n$ there exist positive integers $l_i$ and $m_i$ such that $[G:C_G(x_i^{l_i})] = m_i$. Let $l = kl_1l_2\cdots l_n$ and $m = max\{m_1, m_2 \dots, m_n\}$. Taking into account that $M$ is abelian we deduce that any element of $M^l$ has at most $ms$ conjugates in $G$. Therefore $Pr(M^l,G)\geq\frac{1}{ms}$ and, by Proposition \ref{ds}, there is an open normal subgroup $T\leq G$ and an open subgroup $B \leq M^l$ such that the order of $[T,B]$ is finite. 

Suppose that the order of $[T,B]$ is $f$. Since $B$ normalizes $[T,B]$, it follows that $B^{(f-1)!}$ centralizes $[T,B]$  Choose $t\in T$ and $x\in B^{(f-1)!}$. 
 We have $[x^f,t]=[x,t]^f=1$. Thus $B^{f!} \leq C_G(T)$. Obviously, $K/B^{f!}$ is torsion so the proof is complete \end{proof} 
 
\begin{corollary}
Let $G$ be a finitely generated compact group. Then $Pr(\langle x\rangle , G)> 0$ for every $x\in G$ if and only if $G$ is virtually abelian.\end{corollary}
\begin{proof} As mentioned in the introduction, $Pr(\langle x\rangle , G)> 0$ for every $x\in G$ if and only if $G$ is virtually central-by-torsion. Taking into account that finitely generated compact torsion groups are finite \cite{ze} the result is straightforward.
\end{proof}

\begin{corollary}
Let $G$ be a compact group, and let $G_0$ be the connected component of identity in $G$. Suppose that $Pr(\langle x\rangle , G)> 0$ for every $x\in G_0$. Then $G$ has a normal profinite subgroup $D$ such that $G_0D$ is open in $G$ and $G_0\leq Z(G_0D)$. 
\end{corollary}
\begin{proof} By Theorem \ref{main1} $G$ has an open normal subgroup $T$ such that $G_0/C_{G_0}(T)$ is torsion. Since $G_0$ is divisible, we deduce that $G_0\leq Z(T)$. The structure of compact groups in which the identity component is central is determined in \cite[Lemma 3.5]{hr1}. Therefore $T=G_0\Delta$, where $\Delta$ is a profinite subgroup that is normal in $T$. Set $D=\prod_{x\in G}\Delta^x$ and observe that $G_0D$ is open in $G$ and $G_0\leq Z(G_0D)$.
\end{proof}

\section{On the $p$-structure of a profinite group} 

A multilinear commutator word is a group-word obtained by nesting commutators, but using always different variables. Such words are also known as outer commutator words.  For example, the word $[[[x_1, x_2], [[x_3, x_4], x_5]], x_6]$ is a multilinear commutator, while the Engel word $[[[x_1, x_2], x_2], x_2]$ is not. Among well-known multilinear commutator words there are simple commutators $[x_1, . . . , x_k]$, for which the corresponding verbal subgroups are the terms of the lower central series. Another distinguished sequence of multilinear commutators is formed by the derived words, which are defined recursively as $\delta_0 = x_1, \delta_k =[\delta_{k -1}(x_1,...,x_{2^{k-1}}),\delta_{k-1}(x_{2^{k-1}+1},...,x_{2^k})]$; the corresponding verbal subgroups are the terms of the derived series.

Let $w$ be a multilinear commutator word. For a (profinite) group $G$, we denote by $G_w$ the set of all values of $w$ on elements of $G$ and by $w(G) = \langle G_w\rangle$ the corresponding (closed) verbal subgroup. Let $P$ be a subgroup of a group $G$. Following \cite{khushu14} we denote by $W_G(P)$ the subgroup of $P$ generated by all elements of $P$ that are conjugate in $G$ to $w$-values on elements of $P$, that is, $$W_G(P) = \langle P_w^G \cap P\rangle.$$ \noindent When it causes no confusion, we write $W(P)$ in place of $W_G(P)$.

Let $p$ be a prime number. A profinite group $G$ is pro-$p$-soluble if it has a characteristic series of finite length in which each factor either is pro-$p$ or pro-$p'$. The minimal number (denoted by $l_p(G)$) of pro-$p$ factors in a series of this kind is called $p$-length of $G$. The profinite group $G$ is said to have finite non-$p$-soluble length (denoted by $\lambda_p(G)$) if it has a characteristic series of finite length in which each factor either is pro-$p$-soluble or is isomorphic to a Cartesian product of nonabelian finite simple groups of order divisible by $p$. In this case $\lambda_p(G)$ is the minimal number of non-pro-$p$-soluble factors in such a series. The parameter $\lambda_p(G)$ was introduced (for finite groups) in \cite{khushu15}.

It was shown in \cite{khushu14} that if $w$ is a multilinear commutator word and $P$ is a Sylow $p$-subgroup of a profinite group $G$ such that $W_G(P)$ is torsion, then
\begin{enumerate}
\item $\lambda_p(G)$ is finite;
\item If $G$ is pro-$p$-soluble, then $l_p(G)$ is finite.
\end{enumerate}
In the case where $w=x$ and $p$ is odd the above results were obtained in Wilson \cite{wilson83}.

The above results have the following probabilistic generalization. 


\begin{theorem}\label{length} Let $w$ be a multilinear commutator, $p$ a prime, and $P$ a Sylow $p$-subgroup of a profinite group $G$ such that $Pr(\langle x\rangle, G)>0$ for every $x\in W_G(P)$. Then $\lambda_p(G)<\infty$. If $G$ is pro-$p$-soluble, then $l_p(G)<\infty$.
\end{theorem}

The following lemma is taken from \cite{khushu14}. 

\begin{lemma}\label{wgp} Let $w$ be a multilinear commutator, $p$ a prime, and $P$ a Sylow $p$-subgroup of a profinite group $G$.
\begin{enumerate}
\item If $G_1\leq G$ and $P_1\leq P$, then $W_{G_1}(P_1)\leq W_G(P)$;
\item If $N$ is a normal subgroup of $G$, then $W_{G/N}(PN/N)=W_G(P)N/N$.
\end{enumerate}
\end{lemma}

\begin{proof}[Proof of Theorem \ref{length}.] Set $K=W_G(P)$. By Theorem \ref{main1} the group $G$ contains an open normal subgroup $T$ such that $K/C_K(T)$ is torsion. Let $Z=Z(T)$. Obviously, the image of $K$ in $G/Z$ is torsion. In view of Lemma \ref{wgp} the aforementioned results from \cite{khushu14} imply that 
\begin{enumerate}
\item $\lambda_p(G/Z)$ is finite;
\item If $G/Z$ is pro-$p$-soluble, then $l_p(G)$ is finite.
\end{enumerate}
Since the subgroup $Z$ is abelian, the theorem follows.
\end{proof}

\section{Centralizers of coprime automorphisms of profinite groups}

We say that a continuous automorphism $\phi$ of a profinite group $G$ is coprime if $\phi$ has finite order which is coprime to the orders of elements of $G$ (understood as Steinitz numbers). The following lemma is an obvious extension to profinite groups of a well-known fact on coprime automorphisms of finite groups (see for instance \cite[Theorem 6.2.2]{gor}). 
\begin{lemma}\label{gor}
Let $G$ be a profinite group admitting a coprime automorphism $\phi$ and let $N$ be a $\phi$-invariant normal subgroup of $G$. Then $C_{G/N}({\phi}) = C_G(\phi)N/N$. 
\end{lemma}

In what follows $A^{\#}$ stands for the set of nontrivial elements of a group $A$. 
Let $p$ be a prime. If $G$ is a finite $p'$-group admitting a noncyclic group $A$ of automorphisms of order $p^2$ such that $C_G(\phi)$ has exponent dividing $d$ for every $\phi \in A^{\#}$, then the exponent of $G$ is bounded in terms of $d$ and $p$ only \cite{ks exponent}. If $G$ is acted on by an elementary abelian automorphism group $A$ of order $p^3$ such that the exponent of the commutator subgroup $C_G(\phi)'$ divides $d$ for every $\phi \in A^{\#}$, then the exponent of $G'$ is $(d,p)$-bounded \cite{gushu}.

Non-quantitative profinite variations of these results were obtained in \cite{shuexponent} and \cite{aap}, respectively:
\begin{theorem}\label{helpp}
Let $p$ be a prime and $G$ a pro-$p'$ group admitting an elementary abelian $p$-group of automorphisms $A$.
\begin{enumerate}
 \item If $A$ is noncyclic and $C_G(\phi)$ is torsion for all $\phi \in A^{\#}$, then $G$ is torsion.
 \item If $A$ is of rank at least three and the commutator subgroup $C_G(\phi)'$ is torsion for all $\phi \in A^{\#}$, then $G'$ is torsion.
  \end{enumerate}
\end{theorem}

 Now we are able to supply probabilistic variants: 

\begin{theorem}
Let $p$ be a prime and $G$ a pro-$p'$ group admitting an elementary abelian $p$-group of automorphisms $A$.
\begin{enumerate}
 \item If $A$ is noncyclic and $Pr(\langle x\rangle , G) > 0$ for every $x \in C_G(\phi)$ and $\phi \in A^{\#}$, then $G$ is virtually central-by-torsion.
 \item If $A$ is of rank at least three and $Pr(\langle x\rangle , G) > 0$ for every $x \in C_G(\phi)'$ and $\phi\in A^{\#}$, then $G'$ is virtually central-by-torsion.
  \end{enumerate}
\end{theorem}

\begin{proof} Part (1). Assume that $A$ has rank two and $Pr(\langle x\rangle , G) > 0$ for every $x \in C_G(\phi)$ and $\phi \in A^{\#}$. Let $A_1,\dots,A_{p+1}$ be the maximal subgroups of $A$ and write $K_i$ for $C_G(A_i)$. Applying Theorem \ref{main1} conclude that there are open normal subgroups $T_1,\dots,T_{p+1}$, such that $K_i/C_{K_i}(T_i)$ is torsion for $i=1,\dots,p+1$. Let $T = \bigcap_{i=1}^{p+1} T_i$ and $Z=Z(T)$. Then, by Lemma \ref{gor}, the centralizers $C_{G/Z}(A_i)$ are torsion. In view of Theorem \ref{helpp} (1) we conclude that $G/Z$ is torsion, as required. 

The proof of Part (2) is similar to the above. Assume that the rank of $A$ is at least 3 and write $K_i$ for for the commutator subgroup $C_G(A_i)'$, where $A_1,\dots,A_s$ are the maximal subgroups of $A$. Then proceed as above and use Theorem \ref{helpp} (2) in place of Theorem \ref{helpp} (1).
\end{proof}

\end{document}